\documentclass[12pt]{article}

\usepackage{fullpage}
\usepackage{graphicx}
\usepackage[utf8]{inputenc}
\usepackage{amsmath}
\usepackage{amssymb}
\usepackage{amsthm, xcolor}
\theoremstyle{plain}% default
\newtheorem{thm}{Theorem}
\newtheorem{lem}[thm]{Lemma}

\newtheorem{cor}{Corollary}

\theoremstyle{definition}
\newtheorem{defn}{Definition}
\newtheorem{conj}{Conjecture}

\newtheorem{ques}{Question}

\newtheorem {note}{Observation}

\newcommand{\N}{\mathbb{N}}
\newcommand{\M}{{\cal M}}
\newcommand{\G}{{\cal G}}
\renewcommand{\H}{{\cal H}}
\renewcommand{\P}{{\cal P}}

\newcommand{\A}{{\cal A}}

\newcommand{\X}{{\cal X}}
\newcommand{\ms}{{\ensuremath \star}}
\renewcommand{\geq}{\geqslant}
\renewcommand{\leq}{\leqslant}

\newcommand{\x}{{\boldsymbol x}}
\newcommand{\y}{{\boldsymbol y}}
\newcommand{\z}{{\boldsymbol z}}
\newcommand{\m}{{\boldsymbol m}}
\newcommand{\0}{{\boldsymbol 0}}
\newcommand{\NN}[1]{{\ensuremath {\N_0^{#1}}}}
\newcommand{\nn}[1]{{\ensuremath {\N_0^{#1} \setminus\{\0\}}}}

\title{A mis\`ere-play $\star$-operator}
\author{Matthieu Dufour\footnote{Dept. of Mathematics, Universit\'e du Qu\'ebec \`a Montr\'eal, Montr\'eal, Qu\'ebec H3C 3P8, Canada, dufour.matthieu@uqam.ca} \and Silvia Heubach\footnote{Dept. of Mathematics, California State University Los Angeles, Los Angeles, CA 90032, USA, sheubac@calstatela.edu} \and Urban Larsson\footnote{Dept. of Mathematics and Statistics, Dalhousie University, Halifax, Canada,  urban031@gmail.com}}
\begin{document}
\maketitle
\begin{abstract}
We study the $\star$-operator (Larsson et al, 2011) of impartial vector subtraction games (Golomb, 1965). Here we extend the notion to the mis\`ere-play convention, and prove convergence and other properties; notably more structure is obtained under mis\`ere-play as compared with the normal-play convention (Larsson 2012).  
\end{abstract}

\section{Introduction}

The notion of \emph{vector subtraction games} was introduced by Golomb \cite{G66}, motivated by methods in computer science. \iffalse(concerning shift-registers)\fi Then, much later the game family reappeared \cite{DR10} under a different name (invariant subtraction games) and now the motivation was a conjecture in number theory. \iffalse(concerning complementary Beatty sequences).\fi The proposed problem was solved \cite{LHF11} by introducing the normal-play \emph{$\star$-operator} on the class of games, and subsequently, some very general properties of this $\star$-operator were discovered \cite{L12}. All this work was done using the so-called normal-play convention for impartial combinatorial games \cite{BCG82}. Here we introduce the $\star$-operator under the \emph{mis\`ere-play} convention and prove some general properties. Let us begin by using an example of a game in one dimension (those are usually just called subtraction games).

Imagine two players who alternate in removing tokens from a single heap, subject to the rules that either 4 or 9 tokens be removed, and that if you cannot move, you win (\emph{mis\`ere-play}). In this particular game the first player to move wins if there are less than 4 tokens in the pile, because these positions are terminal, and if there are between 4 and 7 tokens, then the other player wins. 
By a recursive procedure, one computes the pattern of P-positions (these are the starting positions from which the current player cannot win given optimal play). The initial pattern of P-positions is shown in the first line of Figure~\ref{fig:exgame} and the sequence is periodic as illustrated (on the 0th line we show the allowed moves of this game). 

Since the underlying structure of the moves and the P-positions is the same (the non-negative integers), one can play a new game where the P-positions of the first game are used as moves in the new game\iffalse, and this rule defines the $\star$-operator\fi. The new set of moves is then $$\G^1=\{4,5,6,7,12,17,18,19,20,25,\ldots\}.$$ The P-positions of this new game are shown on the second row of Figure~\ref{fig:exgame}. By iterating this process we get a sequence of  games where the moves in the next game consist of the P-positions of the previous game (this is the $\star$-operator to be defined formally below). 

\begin{figure}[hbt]
\begin{center}
\includegraphics[scale=0.80]{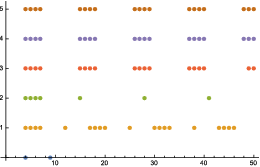}
\end{center}
\caption{The move sets of a sequence of games arising from the initial game with move set  $\G^0=\{4, 9\}$. The values at level $i$ represent the P-positions of the game whose moves are listed on level $i-1$. \iffalse They become the allowed moves in the $i$th game.\fi}
\label{fig:exgame}
\end{figure}
In Figure~\ref{fig:exgame}, the games shown on rows 4 and 5 have the same moves, and one of the results in this paper is  that the sequence of games converges to a limit game, for any choice of the initial set of moves, and in any dimension.

We also show that the limit game is the same for any two games (in the same dimension) if the set of smallest (in the natural partial order) moves is the same.  \iffalse That the set of smallest moves remains fixed for any number of applications of the $\star$-operator is direct by definition of the mis\`ere-play $\star$-operator and this property is fundamental in  distinguishing the $\star$-operator's  behavior from that of normal-play.\fi Moreover, the limit game is reflexive, and we show that it can be defined (non-recursively) by a simple ``sum-set'' rule. This is the third main result of this paper. We have started preliminary work on two dimensions and have obtained partial structure results, which we will discuss in the section on future work. 
 
 We now explain the basic concepts and definitions. Let $\N$ denote the positive integers, and $\N_0$ the non-negative integers. Unless otherwise stated, $\M$ will be a mis\'ere play game on $d\in \N$ heaps (dimension $d$), and we use calligraphy notation for sets when we want to indicate that  we think of a subset of vectors as a game. All games we consider are impartial and of the following form, e.g. \cite{ DR10,G66}.

\begin{defn}\label{def:game}
Let $d\in \N$, and let $\M \subseteq \N^d_0$ be the set of moves. In the  {\em $d$-dimensional vector subtraction game} $\M$,  a player can move from position $\x\in \N^d_0$ to position $\y\in \N^d_0$ if $\x-\y\in \M$.   A position $\y$ for which $\x-\y \in \M$ is called an {\em option of} $\x$.  We consider the mis\`ere-play version of the game, that is, a player who cannot move wins. \end{defn}

%\begin{defn}\label{def:game}
%Let $d\in \N$. Then $\M \subseteq \N^d_0$ describes a class of $d$-dimensional {\em vector subtraction games} for which a player can move from any given $\x\in \N^d_0$ to $\y\in \N^d_0$ if $\x-\y\in \M$.   A position $\y$ for which $\x-\y \in \M$ is called an {\em option of} $\x$. We consider the standard impartial play convention with alternating play in the mis\`ere-play version of the game, that is, a player who cannot move wins. \end{defn}

Note that when we talk about a game, we refer to its rule set or {\em subtraction set}, and we are interested to determine for all positions whether they are P- or N-positons. We are not interested in finding optimatial strategies from a starting position, but rather want to investigate  the patterns of the set of P-positions.

Since our games are multidimensional, we use the natural partial order on $\N^d_0$, namely $\x \preceq \y$ if and only if $x_i \leq y_i$ for $i = 1, \ldots, d$, and $\x \prec \y$ if and only  $\x \preceq \y$  with strict inequality holding for at least one component.

\begin{defn}
A non-empty subset $I$ of a partially ordered set $(X,\preceq)$ is a {\em lower ideal} if for every $\x \in I$, $\y \preceq \x$ implies that $\y \in I$.
\end{defn}

We denote the set of terminal positions of the game $\M$ by $T_{\M}$. By definition of a vector subtraction game, $T_{\M}$ is the set of all $\x$ smaller than or unrelated to every $\m\in \M$, that is, $T_{\M}=\{\x\not \succeq \m \mid \m\in \M\}$. Of course, if $\boldsymbol 0\in \M$ then  $T_{\M}=\varnothing$. Moreover, since we play the mis\`ere version, we have the following observation.

\begin{note}\label{obs:1}
For any game $\M$, in any dimension, the set of terminal positions is a lower ideal, and, moreover, $ T_\M\subseteq N(\M)$.
\end{note}

%We denote vectors in bold, but when it is clear that we are in one dimension (play on a single heap), then we will use regular font to denote positions and moves. 

It is well-known that for impartial games without cycles (that is, no repeated game positions), there are exactly two {\em outcome classes}, called N and P~\cite{BCG82}. In mis\`ere-play, they are characterized as follows: a position is an N-position if it has no option, or if there is at least one P-position in its set of options. Otherwise, a position is a P-position. In other words, a position is a P-position if and only if its set of options is a non-empty set of N-position. 
We denote the set of N-positions of a mis\`ere-play game $\M$ by $N(\M)$, and the set of P-positions by $P(\M)$. %In this paper we also allow the possibility of $\boldsymbol 0$ as a move, so loops are possible (but not cycles !). The reader may argue that we need to use a more elaborate definition of outcomes to include drawn positions. It is not necessary, because if $\boldsymbol 0\in \M$ then all positions are drawn and otherwise we use the standard definition of P and N. 

Note that in Definition~\ref{def:game}, we allow $\M=\varnothing$ and also the case $\0\in \M$  (that is, a pass move is allowed). If $\0 \in \M$, then each position can be repeated so the outcome is a draw, and hence $P(\M)=\varnothing$. This trivial draw game was  originally included in the definition of normal-play vector subtraction games by Golomb \cite{G66}\footnote{He also restricted the set of terminal positions to contain only $\boldsymbol 0$, a definition not used in connection with the $\star$-operator.}. It is not very interesting from a game player's perspective, but from a theoretical point of view, as we will see, there is no reason to exclude it. Similarly, if $\M=\varnothing$, then $P(\M)=\varnothing$, because all positions are N-positions due to the mis\`ere convention. 

On the other hand, if $ \0 \not\in\M$, then we get a recursive definition of the outcomes of all positions from the characterization of N- and P-positions above, and by Observation~\ref{obs:1}, recurrence starts with N-positions. 
Moreover observe that any smallest move $\0 \neq \m \in \M$ is a P-position, so in this case $P(\M) \neq \varnothing$. In fact, each game $\M$ has a unique set of minimal elements which we denote by $\min({\M})$,\footnote{In one dimension $\min({\M})$ consists of a single value and we sometimes abuse notation and write the minimal number instead of the set. If $\M =\varnothing$ then we define $\min \M =\varnothing$.} and we have the following fundamental observation.  %If $\min (\M) = \{\0\}$, then $\0 \not \in P(\M)$. Excluding this case we get the following fundamental observation.

\begin{note}\label{obs:2}
For any game $\M$, in any dimension, if $\0\ne \min(\M)$, then $\min (\M)\subseteq P(\M)$.
\end{note}

%Note that, by Observation~\ref{obs:1}, if $\boldsymbol 0\not \in \M$, then recurrence of outcomes starts with N-position(s). 

Since the underlying structure of moves and P-positions is the same (sets of integer vectors), we can iteratively create new games \cite{LHF11,L12}.

\begin{defn}\label{def:star}
Let $\M$ be a game in any dimension. Then $\M^\ms$ is the game with subtraction set $\M^\ms =  P(\M)$.
\end{defn}

This defines the {\em mis\`ere-play $\star$-operator\footnote{Note that the $\star$-operator under mis\`ere rules is the same as the $\star$-operator in normal-play \cite{LHF11,L12}. However, since in mis\`ere-play  $\0$ is never a P-position, the definition simplifies in this case. }} which acts on impartial subtraction games\footnote{The $\star$-operator is in fact an infinite class of operators, one operator for each dimension. However, we will refer to "the" $\star$-operator because the operator acts in the same way in each dimension.}. A P-position in game $\M$ becomes a move in game $\M^\star$ (and an N-position in $\M$ becomes a non-move in game $\M^\star$). We can now study properties of  sequence of games created by repeated application of the $\star$-operator. First we define special sequences of games, obtained by the fixed points of the operator. 

\begin{defn}\label{def:refl}
The game $\M\subseteq \N^d_0$ is {\em reflexive} if $\M =\M^\ms$.
\end{defn}

\begin{defn}
Let $\M^0=\M$ be a game in any dimension, and let $\M^i=(\M^{i-1})^\star$ for $i>0$. The sequence of games $\M^i$ \emph{converges} (with respect to $\ms$) if $\M^\infty=\lim_{i\rightarrow \infty} {\M}^i$ exists.
\end{defn}

Note that due to the recursive definition of the outcomes of an impartial combinatorial game the notion of convergence is `point-wise'. The following lemma is immediate from the definition of reflexivity.

\begin{lem}\label{lem:1}
The game $\M\subseteq \N^d_0$ is reflexive if and only if there is a game $\X$ such that $\M = \X^\infty$.
\end{lem}

\begin{proof} If $\M$ is reflexive, then we may take $\X=\M$, because $\M = \M^\star =\cdots =\M^\infty$. If $\M = \X^\infty$, for some game $\X$, then by definition of a limit game, $\M$ is reflexive.
\end{proof}

\begin{note}\label{obs:Pempty}
We have that $P(\M)=\varnothing$ if and only if $\0 \in\M$ or $\M=\varnothing$. Consequently,  if $\0 \in\M$ or $\M=\varnothing$, then $\M^\infty=\varnothing$.
\end{note}

%The set of P-positions indicates how to play a game well;  knowing it is an essential part in solving a game. The remaining part is to find a good move, but this question is meaningless unless the set P is known. It makes sense to classify games according to the first part of this observation, modulo particular victorious moves (which is a stronger requirement than what is sometimes called ``P-equivalence" because we consider infinite sets of starting positions). The set of reflexive games are of special interest, and we will be interested in classes of games that share specific sets of P-positions.
Vector subtraction games that have the same sets of P-positions have been studied before (see e.g. \cite{LW}). We will be particularly interested in games for which the set of P-positions is a reflexive game, which motivates the following definition. 
\begin{defn}
Given mis\`ere or normal-play convention, we call a {set} of games $G = \{\G_i\}$ $S$-\emph{solvable} if, for all $i$, $P(\G_i) = S$. If $G$ contains all such games (that is, $\H\not\in G$ implies $P(\H)\ne S$), then we say that the set of games $G$ is $S$-\emph{complete}. 
\end{defn}

In the next section we expand on the one heap example from Figure~\ref{fig:exgame}. Then, in Section~\ref{sec:conv}, we show properties related to convergence in any dimension. In Section~\ref{sec:oneheap}, we discuss structure results on one heap. %, including a result of showing a finite number of iterations until convergence. 
In Section~\ref{sec:twoheaps}, we indicate future directions on two heaps of tokens.

\section{One heap examples}\label{ex:1Dlim}
We begin by illustrating our results on reflexive games and their limit behavior via the following examples of play on one heap.

%\begin{exmp}
Figure~\ref{fig:oneDim} shows the result of applying the  $\ms$-operator five times to two different games. On the left, the move set is  $\H^0=\{4,7,11\}$, while on the right, it is $\G^0=\{4,9\}$ (same as in Figure~\ref{fig:exgame}). Note that both sets have the same minimal move, $k=4$.  Figure~\ref{fig:oneDim} suggests that both games converge to the same limit game, which exhibits a  periodic structure: it consists of groups of $k$ consecutive integers, and the smallest values in consecutive groups differ by $13=3\cdot 4-1=3\cdot k-1$. We will show that all games $\M$ have a limit game under the mis\`ere-play $\star$-operator and that the limit game is determined by the set of smallest elements.

\begin{figure}[htb]
\begin{center}
\includegraphics[scale=0.75]{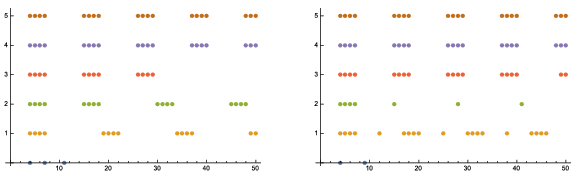}
\end{center}
\caption{The behavior of the  $\ms$-operator for two different games, $\H^0=\{4, 7, 11\}$ and $\G^0=\{4, 9\}$, that have the same minimal move. The values at level $i$ represent the move sets $\H^i$ and $\G^i$, respectively. }
\label{fig:oneDim}
\end{figure}

In proving the convergence result, the approach is to show that the outcome class (move or non-move) of the  smallest position with differing outcome class  in consecutive games will become  ``fixed" in subsequent iterations. Therefore, the set of positions whose outcome class remains unchanged from iteration to iteration increases in each step, and any values already in the set of ``fixed" positions cannot become ``unfixed". Figure~\ref{fig:conv} shows the first five iterations of the game $\G^0=\{4,9\}$. The rectangles identify the smallest elements that differ when comparing $\G^i$ and $\G^{i+1}$. For example, for games $\G^0$ and $\G^1$, the smallest differing element is $x=5$. For $\G^0$, $x=5$ is not a move, but for $\G^1$ (and all subsequent games) it is. Similarly,  the smallest differing element when comparing  $\G^1$ and $\G^2$ is $x = 12$, which is a move in $\G^1$, but then becomes fixed as a non-move in $\G^2$ and subsequent games. For the game $\G^0=\{4,9\}$, the initial set of {\em outcome-fixed} positions is $\{1, 2, 3, 4\}$ (the terminal positions and the smallest move),  $\{ 1, 2, \dots , 11\}$ after the first iteration, then $\{ 1, 2, \dots , 15\}$, and finally $\{ 1, 2, \dots , 47\}$. 

\vspace{0.2in}

\begin{figure}[hbt]
\begin{center}
\includegraphics[scale=0.45]{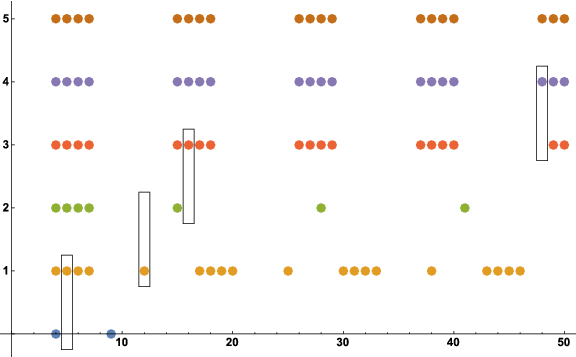}
\end{center}
\caption{Rectangles identifying the smallest elements with differing outcome class in the $i$th and the $(i+1)$st iteration of the game $G^0=\{4,9\}$.}
\label{fig:conv}
\end{figure}
%\end{exmp}

%\silvia{need to adjust this paragraph once final structure of paper has emerged.}

\section{Convergence and reflexivity}\label{sec:conv}
As we have seen, the definition of the $\star$-operator does not depend on the given dimension, and as we will see, neither does its most notable property, convergence to a fixed point, the class of reflexive games being the fixed points of the operator. The following lemma makes this property conceivable.

\begin{lem}\label{lem:mincond} If $\0 \not\in \M$, then $\min(\M) = \min(P(\M))$, and consequently, $\min(\M) = \min(\M^i)$ for all $i \ge 0$, and $\min(\M)= \min(\M^\infty)$ if the limit exists.
\end{lem}

\begin{proof} If $\M=\varnothing$, then $P(\M)=\varnothing = \M$, so the conclusion holds. 
If $\M\neq \varnothing$, let $\m \in \min(\M)$. Then by Observation~\ref{obs:2}, $\m\in\P(\M)$. 
Also, for $\x \prec \m$, $\x \in T_{\M} \subseteq N(\M)$, and therefore, $\m\in\min(P(\M))$. 
Thus, $\min(\M) \subseteq \min(P(\M))$. On the other hand, for $\m' \in \min(P(\M))$, assume $\m' \not\in\min(\M)$. 

There are two possibilities. First, if  for $\m\in\min(\M)$, $\m' \succ \m$, then  Observation~\ref{obs:2} contradicts that $\m'\in\min(P(\M))$ (because $\m\in\P(\M)$). Second, if $\m' \not\succ \m\in\min(\M)$, then $\m'\in T_{\M}$, which contradicts $\m'\in\P(\M)$. 
Therefore, $\m'\in\min(\M)$, which implies that $\min(P(\M)) \subseteq \min(\M)$, so $\min(P(\M)) = \min(\M)$. By definition of the $\star$-operator, we have that $\min(\M) = \min(\M^i)$ for all $i \ge 0$, and the last statement follows  from the definition of the limit game.
\end{proof}

For the $\star$-operator, its definition as well as most of its important properties are independent of the dimension, and it is the main purpose of this section to study these general properties. To emphasize the type of behavior, we introduce the class of \emph{accumulation point operators}. 
\begin{defn}
Let $\Omega$ be a (totally ordered) set, and let $f:\Omega^d \rightarrow \Omega^d$ be an operator defined in any dimension $d\in \N$. Then $f$ is an  \emph{accumulation-point operator} (associated with $\Omega$) if, for any dimension $d\in \N$ and any $X\subseteq\Omega^d$, $\lim_{n\rightarrow \infty} f^n X$ exists, where $f^nX = f\!f^{n-1}X$, for $n>0$ and $f^0=f$.
\end{defn}

In the context of our vector subtraction games,  recall that $\Omega=\mathbb N_0$ includes the case of pass moves. 

\begin{thm} \label{thm:conv}
The mis\`ere-play $\star$-operator is an accumulation-point operator associated with $\mathbb N_0$. That is, for any $d\in \N$, each game $\M \subseteq \mathbb N_0^d$ converges to a (reflexive) limit game $\M^{\infty}$. 
\end{thm}

\begin{proof}
If either $\M=\varnothing$ or $\0\in\M$, then by Observation~\ref{obs:Pempty}, $\M^\infty$ exists. 

Now let $\varnothing\neq\M\subseteq\nn{d}$. Assume that for some $i\geqslant 0$, $$(\M^i\setminus P(\M^i))\cup (P(\M^i)\setminus \M^i)\neq \varnothing,$$ since otherwise  $\M^0=\M^1= \M^\infty$ (by the definition of the star-operator). 
Let $$X(i) = \min((\M^i\setminus P(\M^i))\cup (P(\M^i)\setminus \M^i))$$ be the set of minimal differing elements among moves and P-positions at the $i^{{\rm th}}$ iteration. 
Note that by definition of $X(i)$, if $\z \not\succeq \x$ for all $\x\in X(i)$, then $\z \in (\M^i\cap P(\M^i))\cup ((\M^i)^c\cap (P(\M^i))^c)$, so either  $\z\in \M^j$ for all $j>i$ or $\z\notin \M^j$ for all $j>i$. Now for each $\x\in X(i)$, we consider the following two cases:

Case 1: Suppose $\x\in \M^{i}\setminus P(\M^{i})$.  It suffices to show that $\x\not\in P(\M^{i+1})$, since then $\x\not\in \M^j$ for all $j>i$. Note that $\x$ is not a terminal position because $\x$ is a move. Also, because $\x$ is not a P-position, there is a move  $\m\in\M^i$ such that 
\begin{align}\label{eq:1}
\x - \m = \z\in P(\M^{i}).
\end{align}
However, since $\0\prec\m, \z \prec\x$, then, by definition of $X(i)$, $\m\in\M^{i+1}$ and $\z\in P(\M^{i+1})$, which, by equation~(\ref{eq:1}), implies that $\x\not\in P(\M^{i+1})$, as desired.\footnote{An example of this case is $x=12\in X(1)$ in Figure~\ref{fig:conv}.}

Case 2: Suppose that $\x\in P(\M^{i})\setminus \M^{i}$. It suffices to show that $\x\in P(\M^{i+1})$, since then $\x\in \M^j$ for all $j>i$. Let's assume to the contrary that $\x\not\in P(\M^{i+1})$. Then there exists a move $\m\in\M^{i+1}$ such that $\x - \m = \z\in P(\M^{i+1})$. But then $\m\in P(\M^i)=\M^{i+1}$, by definition of the $\star$-operator, and $\z\in \M^i\cap P(\M^i)$, by definition of $X(i)$. Therefore, in the game $\M^i$ we have a move $\z$ from a P-position $\x$ to the P-position $\m$, a contradiction, so $\x\in P(\M^{i+1})$.\footnote{Examples of this case are $x=5\in X(0)$, $x=16\in X(2)$, and $x=48\in X(3)$ in Figure~\ref{fig:conv}.} 
\end{proof}

 We now characterize reflexive games via a ``sum-set'' property.%, which indicates connections with additive number theory. %(for more on this topic see \cite{DHL}).%see also Section~\ref{sec:add} (and Section~\ref{sec:normal} for a glimpse on an analogous situation in normal-play). 

\begin{defn}
 Suppose that $A, B\subseteq \NN{d}$. Then $A+B=\{{\boldsymbol a}+{\boldsymbol b}\mid {\boldsymbol a}\in A, {\boldsymbol b}\in B\}$.   
\end{defn}

\begin{thm}\label{thm:sumset}
Let $A\subseteq \N_0^d$. Then the game $\A$ with move set $A$ is reflexive if and only if 
\begin{align}\label{eq:sumset}
A+A=A^c\setminus T_\A,
\end{align}
where $A^c$ denotes the complement of $A$ with respect to $\N_0^d$.
\end{thm}

\begin{proof} If $A = \varnothing$, then all positions  are terminal N-positions, so $T_{\A} = \N^d_0$ and $P(\A)=\varnothing=A$.  Thus $\A$ is reflexive and \eqref{eq:sumset} holds. 

Next we assume that $A$ is non-empty. If $\0\in A$, then because $\0\not\in P(\A)$, $\A$ is not reflexive. On the other hand, $\0 \in A+A $, but $\0\not\in\A^c$, so \eqref{eq:sumset} does not hold and the claim is true in this case also. 

Now we assume that $A$ is non-empty and $\0\not\in A$. Note that for any such game $\A$, we have that for any non-terminal position $\x \in N(\A)\setminus T_\A$ there is a move $\m=\x-\z \in A$ that leads to a P-position $\z\in P(\A)$. Therefore, $ N(\A)\setminus T_\A \subseteq A+P(\A)$.

On the other hand, since a move from a P-position cannot result in a  P-position, for any $\z\in P(\A)$ and any move $\m \in A$, $\m+\z=\x \in N(\A)\setminus T_\A$. Thus, $A+P(\A) \subseteq N(\A)\setminus T_\A$, and we have
\begin{align}\label{eqforgen}
A+P(\A) = N(\A)\setminus T_\A. 
\end{align}

We now prove that $\A$ is reflexive if and only if~\eqref{eq:sumset} holds.  \\
``$\Rightarrow$"  If $A$ is a reflexive, then $P(\A)=A$ and $N(\A) =A^c$, so \eqref{eqforgen} reduces to \eqref{eq:sumset}.\\
``$\Leftarrow$"   Let $B=P(\A)$, so we need to prove that $B=A$. Assume to the contrary that there is $\x\in A^c\cap B$.  Then $B\subseteq A^c$, because otherwise, there would exist  $\z\in A\cap B$ and, by \eqref{eq:sumset}, a move $\m = \x - \z\in A$ from P-position $\x\in A^c\cap B\subset A^c\setminus T_\A$ to P-position $\z\in A\cap B$. So, $B \subseteq A^c$, or equivalently, $ A  \subseteq B^c=N(\A)$. However, in mis\`ere-play, a smallest move (which exists by assumption) is always a P-position, which contradicts that $A \subseteq N(\A)$, and so $A^c \cap B = \varnothing$. Therefore,   $B \subseteq A$.

It remains to prove that $A \subseteq B$,  or equivalently,  $B^c \subseteq A^c$. Let $\x \in B^c$. Note that $T_{\A} \subseteq A^c\cap B^c$ because terminal positions are neither moves nor P-positions. Thus,  we assume without loss of generality that $\x \in B^c\setminus T_\A$, that is, $\x$ is a non-terminal N-position. By  \eqref{eqforgen}, there is a move $\m=\x-\z \in A$ from $\x$ to $\z \in B \subseteq A$. Since both  $\z$ and $\m$ are in $A$,  then by assumption \eqref{eq:sumset} we have that  $\m+\z =\x \in A^c\setminus T_\A$. Since $\x\notin T_\A$, we must have that $\x \in A^c$, which completes the proof.
\end{proof}

Using the sum-set property of Theorem~\ref{thm:sumset}, we completely characterize the limit games.  There is exactly one reflexive limit game for each set of minimal moves, that is, the set of minimal elements uniquely determines the limit game.

\begin{thm}\label{thm:general}
Let $\M$ and $\G$ be non-empty games. Then $\M^\infty =\G^\infty \iff \min(\M)=\min(\G)$.
\end{thm}

\begin{proof}
``$\Rightarrow$"  If $\M^\infty = \G^\infty=\varnothing$, then by Observation~\ref{obs:Pempty},  $\{\0\}=\min(\M)= \min(\G)$ since both $\M$ and $\G$ are non-empty. If $\M^\infty$ and $ \G^\infty$ are non-empty games, then by Observation~\ref{obs:2}, $\0 \not\in \M \cap \G$, and by Lemma~\ref{lem:mincond},  we have  $\min(\M)=\min(\M^\infty)=\min(\G^\infty)=\min(\G)$ as claimed. 

``$\Leftarrow$" 
 If $\{\0\} = \min(\M) = \min(\G)$, then $\M^\infty=\G^\infty=\varnothing$ by Observation~\ref{obs:Pempty}. If $\{\0\} \neq \min(\M)=\min(\G)$, then by Lemma~\ref{lem:mincond},  $ \min(\M^\infty)=\min(\G^\infty)$ and $T_{\M^\infty}=T_{\G^\infty}$. We need to show that $\M^\infty=\G^\infty$. Assume to the contrary that there is a smallest differing element $$\x=\min(\G^\infty\setminus \M^\infty\cup \M^\infty \setminus \G^\infty).$$ 
 Without loss of generality we may assume that $\x\in \G^\infty\setminus \M^\infty$. Be definition of $\x$, $\x \not\in \M^\infty$. Also, $\x \succeq \m \in \min(\G^\infty)=\min(\M^\infty)$, so $\x \not \in T_{\M^\infty}$, that is, $\x \in (\M^\infty)^c\setminus T_{\M^\infty}$. Since $\M^\infty$ is reflexive, by  Theorem~\ref{thm:sumset}, there must be $\0\neq \y, \z \in\M^\infty$ such that $\y+\z=\x$. However, since $\y,\z \prec \x$, by minimality of $\x$, we have $\y,\z \in \G^\infty$. Applying  Theorem~\ref{thm:sumset} to $\G^\infty$ now implies that $\x \in (\G^\infty)^c\setminus T_{\G^\infty}$, a contradiction. Thus $\M^\infty=\G^\infty$.
\end{proof}

 Theorems~\ref{thm:conv} and~\ref{thm:general} confirm what was suggested in Figure~\ref{fig:oneDim}; the games converge to the same limit game. Now the question becomes: what do limit games `look like'? We will completely answer this question in the next section for games on one heap, and then in the final section, we sketch some of the observed behavior for two heaps (see also \cite{CDHL}).

Both the mis\`ere-play $\star$-operator and the normal-play $\star\star$-operator converge in any dimension, but the properties of the fixed points are not the same.  Our results imply that the mis\`ere-play convergence is \emph{stable} in the following sense.

\begin{cor}
Let $\M$ be a reflexive game in any dimension, and let $Y$ be a finite set of vectors in the same dimension. For almost all \emph{perturbations} of the form $\M_Y = (\M\setminus Y) \cup (Y\setminus \M)$, then $\M^\infty = {\M_Y}^\infty $. 
\end{cor}

\begin{proof}
This is  a consequence of Theorems~\ref{thm:conv} and \ref{thm:general}. 
\end{proof}

\section{A characterization of limit games in one dimension}\label{sec:oneheap}

We first consider $d=1$, that is, play on a single heap.  Motivated by the structure of the limiting game in Figure~\ref{fig:exgame}, for any  $k\in \N$,  we define the period $p_k=3k-1$ and let ${\M}_k$ denote the set 
\begin{align*}
\M_k &= \{i p_k+k, \ldots, i p_k+2k-1 \mid i\in \N_0\},
\end{align*}
with  $\M_0=\varnothing$. Note that $k=\min(\M_k)$ for $k \geq 1$. 
By Theorem~\ref{thm:general}, the games in Example~\ref{ex:1Dlim} have the same limit game, and  we will see in Theorem~\ref{thm:onedim} and Corollary~\ref{cor:limit}, that $\H^\infty=\G^\infty=\M_4$.

Since the set $\M_k$ is periodic with period $p_k$, we find it convenient to make our arguments in arithmetic modulo~$p_k$. We denote the set of residuals modulo~$p$ of elements of a set $A$ by $[A]_p$. With this notation, it follows from the definition of $\M_k$ that for $k \geq 1$,
\begin{align}\label{eq:Mk_res}
[\M_k]_{p_k}&=\{k, k+1, \ldots, 2k-1\} \textnormal{    and    } \\ 
 [\M_k^c]_{p_k}&=\{0, 1, \ldots, k-1, 2k, \ldots, 3k-2\} \equiv_{p_k}\{-(k-1),\ldots,k-1\}.\notag
\end{align}

\begin{thm}\label{thm:onedim}
The game $\M\subseteq \N_0$ is reflexive if and only if  $\M={\M}_k$, for some $k\in \N_0$.
\end{thm}

\begin{proof}
``$\Leftarrow$" If $k=0$, then  $\M=\M_0=\varnothing$, which is reflexive by Observation~\ref{obs:Pempty}. Suppose next that $\M=\M_k$ is nonempty and let $k=\min(\M_k)\ge 1$. We show that the game $\M_k$ is reflexive using Theorem~\ref{thm:sumset}. Note that by \eqref{eq:Mk_res}, 
$$[\M_k+\M_k]_{p_k}=[\{2k,\ldots,4k-2\}]_{p_k}=\{2k,\ldots,3k-2, 0,\ldots, k-1\}=[\M_k^c]_{p_k}.$$
Since for any element $m\in \M_k$, $m+m  \ge 2k$ and the terminal positions are given by $T_{\M_k}=\{0,\ldots,k-1\}$, we have that $\M_k+\M_k \subseteq \M_k^c\setminus T_{\M_k}$. On the other hand, let $z\in \M_k^c\setminus T_{\M_k}$, so $z= i\cdot p_k+r$ with $r\in [\M_k^c]_{p_k}$. If $0 \le r \le k-1$, then $i \geq 1$ (because $z$ is not a terminal position), and we can write  $z= x+y$ with  $x=(i-1)p_k+ k+ r \in \M_k$ and $y=2k-1 \in \M_k$.
If $2k \le r \le 3k-2$, then $z=x+y$ with $x=i \cdot p_k+k \in \M_k$ and $y=r-k \in \M_k$. Thus $\M_k^c\setminus T_{\M_k}\subseteq \M_k+\M_k$, so $\M_k$ is reflexive by Theorem~\ref{thm:sumset}. 

``$\Rightarrow$" We show that if $\M \neq \M_k$, then $\M$ is not reflexive. Let $k=\min(\M)$. If $k=0$, then $\M\neq \M_\ell$ for any $\ell$, and furthermore, by Observation~\ref{obs:Pempty}, $\M$ is not reflexive. Now assume that $k>0$, so  $k\in P(\M)$ by Observation~\ref{obs:2}. Assume that there is a positive integer $x= \min(\M_k\setminus \M\cup\M \setminus \M_k)$, that is, $x$ is the smallest value that differs between $\M$ and $\M_k$. Necessarily, $x>k$. 

Suppose first that $x\in \M_k\setminus \M$. Because $x \notin \M$, it suffices to prove that $x\in P(\M) = \M^\ms$ to show that $\M$ is not reflexive. Since $x>k$, there exists $y \in \M_k\cap \M \supseteq \{k\}$ such that $y<x$. For any such $y$, $y\in P(\M_k)$ by reflexivity of $\M_k$.  By minimality of $x$,  $y \in P(\M)$ because the same moves are available from $y$ in both $\M$ and $\M_k$. 
Since $x, y \in \M_k$, we have  $x=i \cdot p_k+r$ and $y=j       \cdot p_k+s$ for some $0 \leq j\leq i$ and $k\leq r,s \leq 2k-1$. Thus $z=x - y = (i-j)\cdot p_k+(r-s)$ with $-k+1\leq r-s\leq k-1$, so $z \not\in \M_k$, and by minimality of $x$, $z \not \in \M$. This implies that there is no move in $\M$ from $x$ to a P-position $y$, so $x \in P(\M)$, which completes
this case. 

Suppose next that $x\in \M\setminus \M_k$. It suffices to prove that $x \not\in P(\M)$ to show that $\M$ is not reflexive. By the minimality of $x$, it suffices to find an option $z$ of $x$ with $z\in P(\M)$,  that is  $z=x-y$ for some $y\in \M$. 
Because $y,z <x$, we have $y,z \in \M_k \cap \M$ due to the minimality of $x$. Since $x\not\in \M_k$, $x=i\cdot p_k+r$ for some $i \ge 0$ and $r \in [\M_k^c]_{p_k}$. If $r \in \{0,\ldots, k-1\}$, let $y=(i-1)p_k+(2k-1)\in \M_k$, otherwise choose $y=i\cdot p_k+k \in \M_k$. In each case, $[x - y]_{p_k} \in  [\M_k]_{p_k}$. This shows that there is a move from $x$ to a P-position $z\in P(\M)$, so $x \not\in P(\M)$, which implies that  $\M$ is not reflexive either in this case. 
Overall, the game $\M$ is reflexive if and only if $\M$ is of the form $\M_k$.
\end{proof}

Now that we have identified a family of games that are reflexive, we will show that these games are the only ones that can occur as limit games.

\begin{cor}\label{cor:limit}
Let $\M \subseteq \N_0$ and let $k=\min(\M)$ if $\M \neq \varnothing$, and $k=0$ otherwise. Then $ \lim_{i\rightarrow \infty}{\M}^i = {\M}_k$.
\end{cor}

\begin{proof} Since the limit game is reflexive, Theorem \ref{thm:onedim} applies, and $\M^\infty=\M_j$ for some $j \in\N$. If $\M=\varnothing$ or $0\in\M$, then $\M^\infty = \varnothing = \M_0$, so the claim is true. If $\M$ is nonempty and $0\notin \M$, then by Lemma~\ref{lem:mincond},   $k =\min(\M)=\min(\M^\infty) $. Since $\min(\M_j) = j$ for $j >0$, the minimum uniquely determines $\M_j$,  so we have that $\M^\infty = \M_k$.
\end{proof}

In conclusion,  for $d=1$ we understand the structure of any limit game -- it is periodic and is completely determined by the minimal move. This result is quite surprising in its simplicity, especially since in the case of normal-play, general formulas for limit games are rare in any dimension, the exceptions consisting of a few    `immediately' reflexive game families \cite{LHF11,L12}. 

Now that we have identified the sets $\M_k$ as the only possible limit games, we answer which games have $\M_k$ as their set of P-positions. 

\begin{thm}\label{thm:min1d}
Let $k\in \N$ and $A_{k} = \{k, 2k-1\}$. Then $P(\X) = M_k$ if and only if $A_k\subseteq X\subseteq M_k$. That is, the set of games $\{\X\mid A_k\subseteq X\subseteq M_k\}$ is $M_k$-solvable and also $M_{k}$-complete.

\end{thm}

\begin{proof} We begin by proving that $P(\A_k) = M_k$. 
Clearly, $T_{\A_k} = \{0,\ldots , k-1\}\subset N(\A_k)$. We compute modulo $p_k = 3k - 1$, and use  (\ref{eq:Mk_res}) to justify that for each $x\in M_k^c\setminus \{0,\ldots , k-1\}$, $x-k\in M_k$ or $x-(2k-1)\in M_k$. Indeed, if $x \in \{0,\ldots , k-1\} \pmod {p_k}$, then $x-(2k-1)\in M_k$, and otherwise $x-k\in M_k$. 
For the other direction we must show that for all $x\in M_k$, both $x-k\in M_k^c$ and $x-(2k-1)\in M_k^c$, and this follows directly by (\ref{eq:Mk_res}). Thus $P(\A_k) = M_k$. 

To prove the statement for a general set $X$ with $A_k\subseteq X\subseteq M_k$, we use  that $P(\M_k) = M_k$  (by Theorem~\ref{thm:onedim}). Hence no move in $\X$ connects any two candidate P-positions in $M_k$. Moreover, since $A_k\subseteq X$, for each candidate N-position we find a move to a candidate P-position using the moves $k$ or $2k-1$.

It remains to prove that  no other sets $X$ have the property $ P(\X)=M_k$, that is, we need to show that if there is $ x\in A_k\setminus X$ or $ x\in X\setminus M_k$, then  $ P(\X) \neq M_k$. Suppose that there is a smallest $x\in A_k\setminus X$, with $M_k = P(\X)$. Then $x=k$ or $x=2k-1$; in the first case, if $k$ is not a move, then $P(\X)=M_k$ implies that  $x<k$ is a terminal N-position, so $k$ as a non-move is also terminal and hence an N-position, a contradiction. Hence assume $k$ is a move, but $2k-1$ is not. Then there is no move from $4k-2\in M_k^c$ to a P-position in $M_k=P(\X)$, contradicting that $M_k$ is the set of P-positions. 

Suppose next that there is a smallest move $x\in X\setminus M_k$ with $P(\X) = M_k$. If $x\in T_{\M_k}$, then $\X$ and $\M_k$ do not have the same P-positions (since $x$ is a P-position in $\X$, but an N-position in $\M_k$). Hence, we must have $x\not\in T_{\M_k}$ and $x \in \{-(k-1),\ldots , k-1\} \pmod {p_k}$. But, for each such $x$, we find two P-positions $y,z\in \{k,\ldots , 2k-1\} \pmod {p_k}$ such that $y-z=x$, which contradicts  $x$ being a move. 
\end{proof}

Given a game $\M$ (in any dimension), we denote the number of iterations of the mis\`ere-play $\star$-operator until the limit game appears for the first time by $\varphi(\M) = \min\{i\mid \M^i =\M^\infty \} \in \N_0\cup \{\infty\}$. For the game $\M=\{k\}$, we derive $\varphi(\M)$. 

\begin{lem}\label{lem:varphi5}
Let $\M=\{k\}$ with $k \geq 2$. Then 
\begin{enumerate}
\item $\M^1 = \{x\mid x\equiv k,\ldots ,2k-1\pmod {2k}\}=[\{ k,\ldots ,2k-1\}]_{2k}$.
\item $\M^2 = \{k,\ldots , 2k-1\}\cup \{4k-1, 6k-1,\ldots\}$.
\item $\M^3= \{k,\ldots , 2k-1\}\cup \{4k-1,\ldots , 5k-2\}\cup \{7k-2, 9k-2,\ldots\}$.
\item $\M^4 = \M_k\cap \{0, \ldots , 10k-3\}$.
\item $\M^5 = \M_k$ for any $k$.
\end{enumerate}
\end{lem}

Figure~\ref{fig:conv-k} illustrates Lemma~\ref{lem:varphi5}  for  $\M=\{4\}$.

\begin{figure}[hbt]
\begin{center}
\includegraphics[scale=0.5]{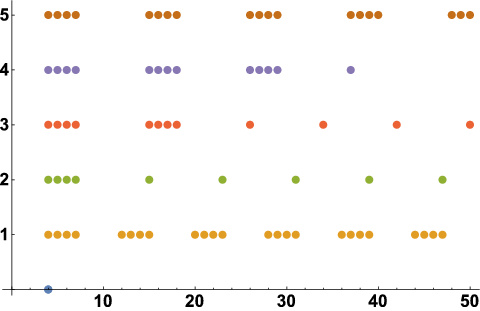}
\end{center}
\caption{The iterations of the mis\`ere-play $\star$-operator for $\M=\{4\}$.}
\label{fig:conv-k}
\end{figure}

\begin{proof} 1. Let $S=[\{ k,\ldots ,2k-1\}]_{2k}$. The terminal positions of $\M$ are given by $T_{\M}=\{0, 1,\ldots, k-1\}\subset S^c$. For any position $x\in S$, the position $x-k \notin S$. Also, for $x\notin S$, the position $x-k \in S$, so $S=P(\M)=\M^1$.\\

\noindent 2. Let $S =  \{k,\ldots , 2k-1\}\cup \{4k-1, 6k-1,\ldots\}$.  The allowed moves are of the form $m = i\cdot 2k + r$ with $k \leq r \leq 2k-1$ and $i \ge 0$. Since $\M^1\cap\M_k=\{0,\ldots, 3k-1\}$, these moves are already fixed as P-positions.  If $3k \leq x \leq 4k-2$, then $x-(2k-1) \in S$, so $x \in N(\M^1)$. If  $x=j\cdot 2k -1$ with $j \ge 2$, then $x-m \in \{0,\ldots,k-1\} \subset S^c$.  Also, for any  $x > 4k-1$ with $x \notin S$, $x = j\cdot 2k+r$ with $0 \le r \le 2k-2$ and $j \ge 2$. Then for $0 \le r < k-1$, $x-m \in S$ for $m = (j-1)\cdot 2k + k+r$, and for $k \le r \le 2k-2$, $x -( r+1) \in S$.\\

\noindent 3. Let $S = \{k,\ldots , 2k-1\}\cup \{4k-1,\ldots , 5k-2\}\cup \{7k-2, 9k-2,\ldots\}$. Note that $\M^2 \cap \M_k= \{0, \ldots , 4k-1\}$. If $x\in \{4k-2,\ldots , 5k-2\}$, then the possible moves from $x$ are of the form $m\in \{k,\ldots , 2k-1\}\cup \{4k-1\}$, which gives $m-x\in \{1,\ldots , k-1\}\cup \{2k+1,\ldots , 4k-2\}\subset N(\M^2)$. Suppose next that $x\in \{5k-1,\ldots , 7k-3\}$. Then there is a move $m\in \{k,\ldots , 2k-1\}$ to a position in the set $\{4k-1,\ldots , 5k-2\}\subset P(\M^2)$, so $x\in N(\M^2)$. If $x\in \{7k-2, 9k-2,\ldots\}$, then one can easily check that there is no move to any $y \in S$. If $7k-2\le x\not\in \{7k-2, 9k-2,\ldots\}$, then for  $(2i-1)k-1 \le x\le (2i)k-2$ and $i \ge 4$,  the move $m=2(i-1)k-1$ will lead to a losing position in $\{k,\ldots, 2k-1\}$, while for $(2i)k-1 \le x\le (2i+1)k-1$, the move leading to a P-position is $m=2(i-2)k-1$.\\

\noindent 4. Note that $\M^3$ is identical with $\M_k$ for positions  $x \le 7k-2$,  so it remains to investigate the case $x> 7k-2$. Here the argument is similar to 3.\\

\noindent 5. This follows from Theorem~\ref{thm:min1d}.\end{proof}

\begin{cor}\label{cor:varphi}
For  $\M=\{k\}$, convergence to the limit set $\M^\infty=\M_k$ occurs in a finite number of steps. In particular, $\varphi(\{0\}) = \varphi(\{1\}) = 1$, and for $k \geq 2$, $\varphi(\{k\}) = 5$.
\end{cor}
\begin{proof}
For $k=0$, it follows from Observation~\ref{obs:Pempty} that $\M^1=\varnothing=\M_0$.  For $k=1$, let $S=\{1,3,5,\ldots\}=\M_1$. Then for $x\in S$, $y=x-1\in S^c$, and likewise, for $x\in S^c$, $y=x-1\in S$, so $P(\{1\})=\M_1$. In both cases, $\varphi(\{k\}) = 1$. For $k \ge 2$,  $\varphi(\{k\}) = 5$ follows by Lemma~\ref{lem:varphi5}.
\end{proof}

We do not yet understand $\varphi(\M)$ for any other case than the one described in Corollary~\ref{cor:varphi}. We have some experimental suggestions in the 2-dimensional case, which leads us to the next section.

\section{Structures in 2 dimensions}\label{sec:twoheaps}
This section is intended as an overview of the behavior in 2 dimensions, and should be regarded as an informal exposition. We indicate experimental similarities and differences with the known structure in one dimension.  
 
In one dimension, all reflexive games have the same geometrical  structure up to rescaling (as demonstrated in Section~\ref{sec:oneheap}). In two dimensions, the geometrical structures of the reflexive games vary much more, even though for certain classes of games we still obtain similar rescaled structures. At the very least, our experiments show that we must distinguish classes of games according to where the minimal moves occur, as they must have different behavior due to Theorem~\ref{thm:general}. That the conjectured behavior is the same within each class is harder to prove in general, but possible to be shown in certain cases.  The following classification scheme is the least required:

\begin{enumerate}
\item The game has only one minimal move
\begin{enumerate}
\item on one of the axes
\item not on an axis
\end{enumerate}
\item The game has exactly two minimal moves
\begin{enumerate}
\item none of the minimal moves is on an axis
\item exactly one of the minimal moves is on an axis
\item both minimal moves are on the axes
\end{enumerate}
\item The game has at least three minimal moves
\begin{enumerate}
\item none of the minimal moves is on an axis
\item exactly one of the minimal moves is on an axis
\item there is a minimal move on each axis
\end{enumerate}
\end{enumerate}

The class 2(c) most closely resembles the one-dimensional case, as the two-dimensional limit game inherits some of its structure from the respective one-dimensional limit games. Figure~\ref{fig:03_40iter} shows the iterations for a game of the form $\M=\min({\M})=\{(k, 0), (0, \ell)\}$, the simplest form of  case 2(c), for $k=4$ and $\ell = 3$. It appears that this game converges to a limit game after seven steps. In addition, after five steps, the behavior along the axes is as described in Theorem~\ref{thm:onedim}.

\begin{figure}[ht!]
\begin{center}
\includegraphics[scale=0.2]{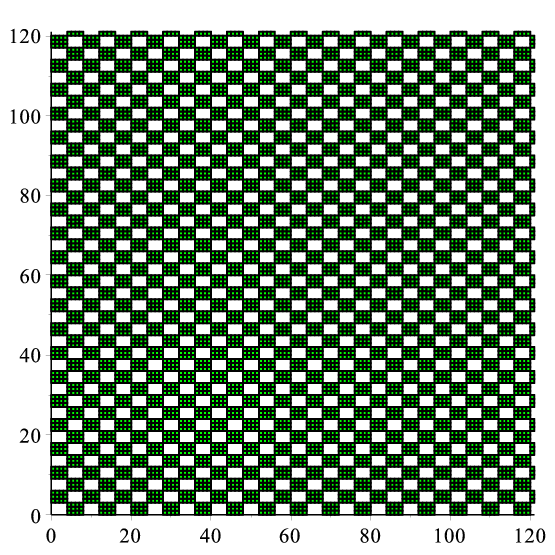}\hspace{1mm}
\includegraphics[scale=0.2]{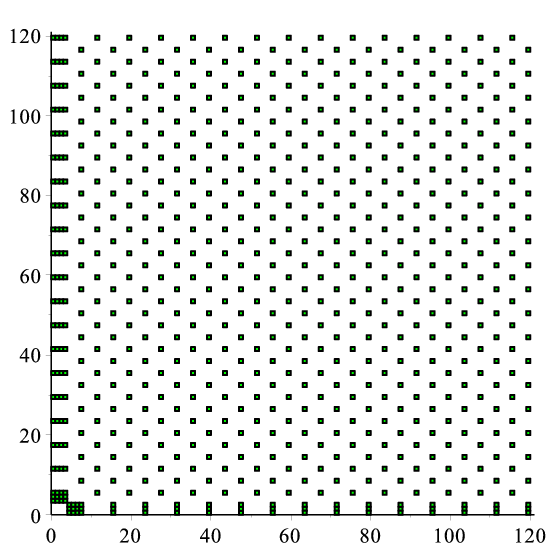}\hspace{1mm}
\includegraphics[scale=0.2]{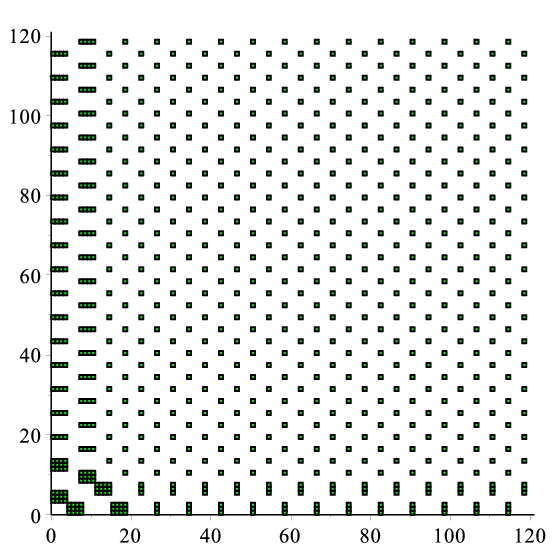}\hspace{1mm}
%\vspace{3 mm}
\includegraphics[scale=0.2]{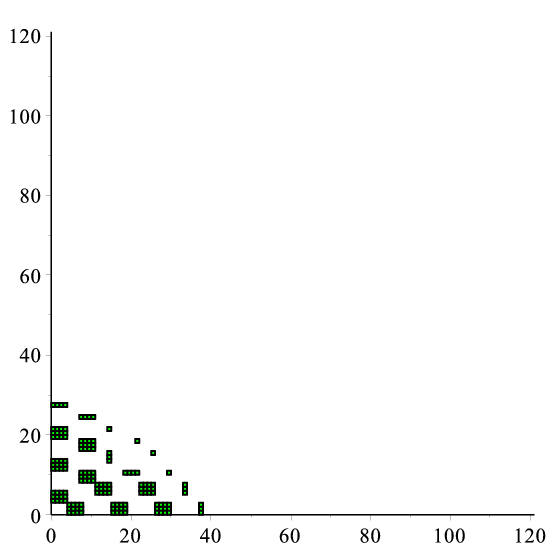}\hspace{1mm}
\includegraphics[scale=0.2]{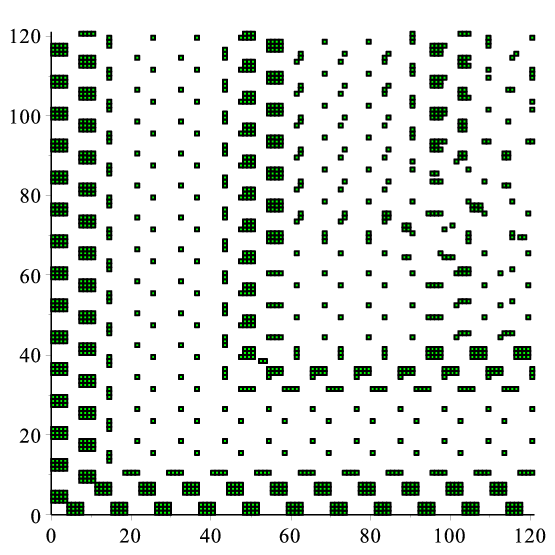}\hspace{1mm}
\includegraphics[scale=0.2]{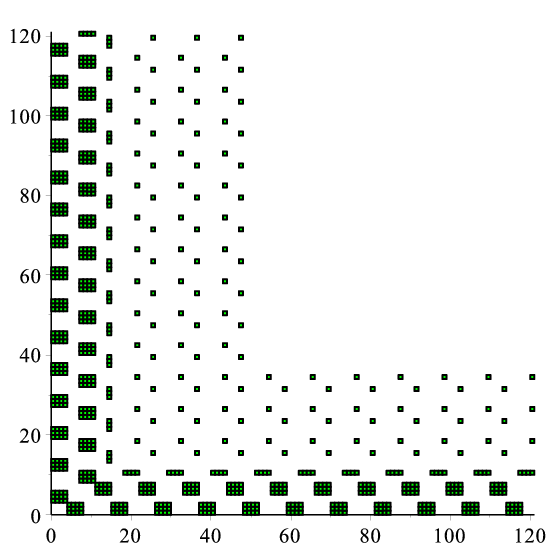}\hspace{1 mm}
\includegraphics[scale=0.2]{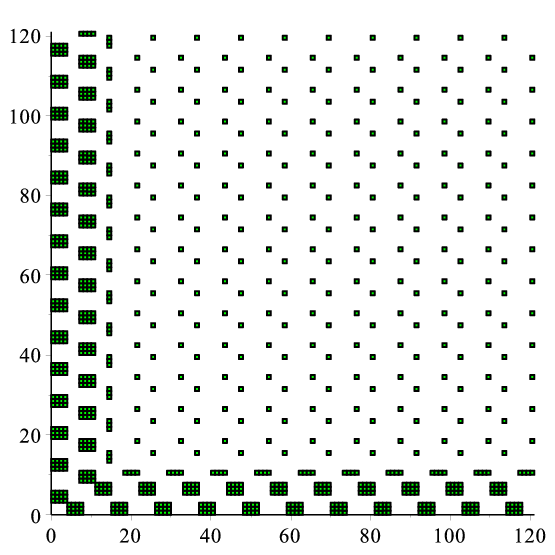}\hspace{1 mm}
\includegraphics[scale=0.2]{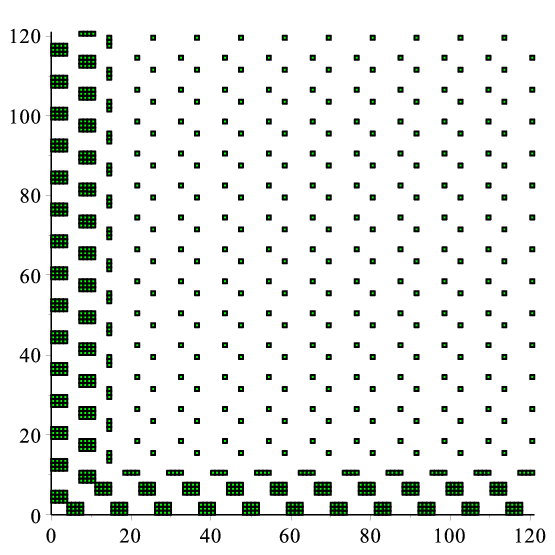}
\end{center}
\caption{Iterations of the mis\`ere-play $\ms$-operator for the game $\M=\{(4,0),(0,3)\}$ where the game shown in the upper left is $\M^\star$. The limit game is reached after 7 steps in this case.}
\label{fig:03_40iter}
\end{figure}

Informally we define $\M_{k, \ell}$ as the type of limit game shown in Figure~\ref{fig:03_40iter} (for $k=4$ and $\ell = 3$). It can be defined in a periodic manner based on $k, \ell$, and the one-dimensional associated periods. We are in the process of proving this game to be reflexive~\cite{CDHL}. Due to the periodic structure of $\M_{k, \ell}$, we know the limit game to be periodic along half lines of rational slopes. The structure of the limit game is generic, but the number of iterations until convergence can vary for this class. 

Computer explorations for games in the other classes (see for example Figures~\ref{fig:3a} and~\ref{fig:3cx}) suggest that all limit games have some type of periodic structure, which leads to the following conjecture.

\begin{conj}
Limit games for all two-dimensional vector subtraction games under the mis\`ere-play $\ms$-operator are ultimately periodic along any line of rational slope.
\end{conj}

%Denote by $\M^\infty =\M_{k, \ell}$ the generic `crystallin' structure as observed in the limit diagrams of Figures~\ref{fig:03_40iter}~and~\ref{fig:invariant}. 
%
%In these cases, the limit game is periodic along any half-line of rational or infinite slope (perhaps with a finite pre-period). Does this hold in general for any class?

Returning to class 2(c), one can ask which games $\A_{k,\ell}$ have the property that $P(\A_{k,\ell})=\M_{k,\ell}$ (see Theorem~\ref{thm:min1d} for the one-dimensional equivalent), and more specifically, whether there is a smallest such game. In Figure~\ref{fig:invariant} we display the `smallest' game discovered so far that has the reflexive game $\M_{k,\ell}$ as its set of P-positions. 

\begin{ques}
See the left most picture in Figure~\ref{fig:invariant}. Is this a generic description of a smallest game with a reflexive game of type 2(c) as its set of P-positions?
\end{ques}

\begin{figure}[ht]
\begin{center}
\includegraphics[scale=0.25]{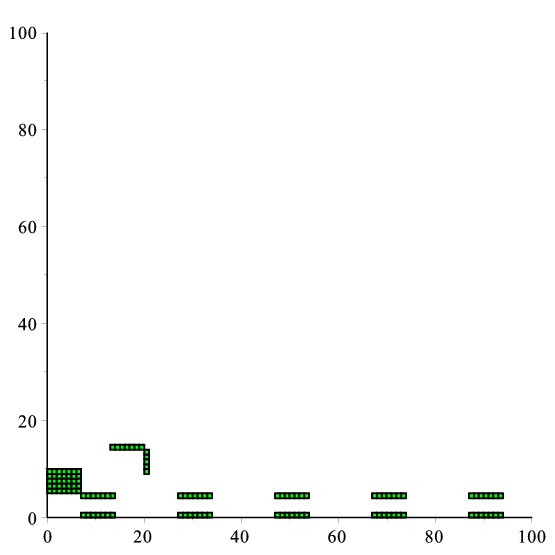}
\includegraphics[scale=0.25]{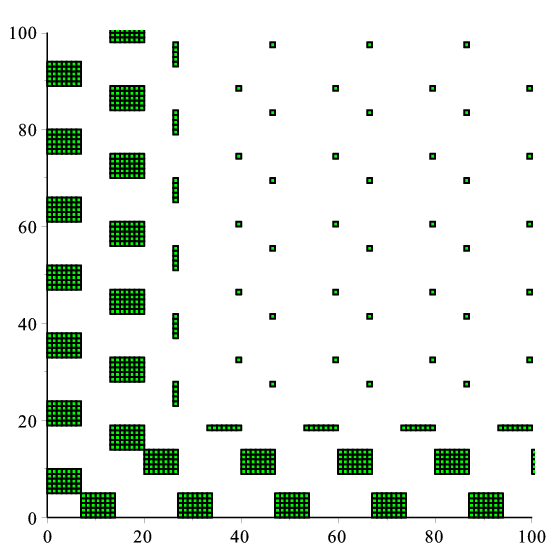}
\end{center}
\caption{The graph on the right represents the P-positions of the game $\A_{k,l}$ shown on the left, with $(k, \ell)=(7,5)$.}
\label{fig:invariant}
\end{figure}

%Possibly, the case 3(c) could provide a counter example to this conjecture; in this case the computer simulations have not yet revealed any solid classification, but as far as we have been able to compute, we have detected periodic behavior. See Figure~\ref{fig:3cx} for a couple of simple examples in this class; after the initial examples observed periods grow fast (in the size of the smallest moves). 
%

We conclude this section with some explorations in cases when there are at least three minimal moves. 
Suppose that $\min (\M)\cap \{(0,x),(x,0)\mid x\in \N\}=\varnothing$, so we are in class 3(a). Then $\varphi(\M)=2$, that is, $\M^{\star\star}$ is reflexive. It is not hard (but somewhat technical) to prove this statement by an explicit description of the generic description of the right most graph. Note also that this `penultimate lower ideal' is already a subset of the second graph.

\begin{figure}[hb]
\begin{center}
\includegraphics[scale=0.2]{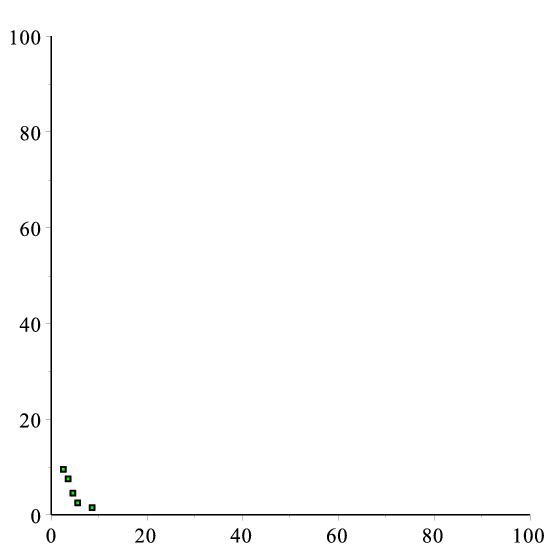}
\includegraphics[scale=0.2]{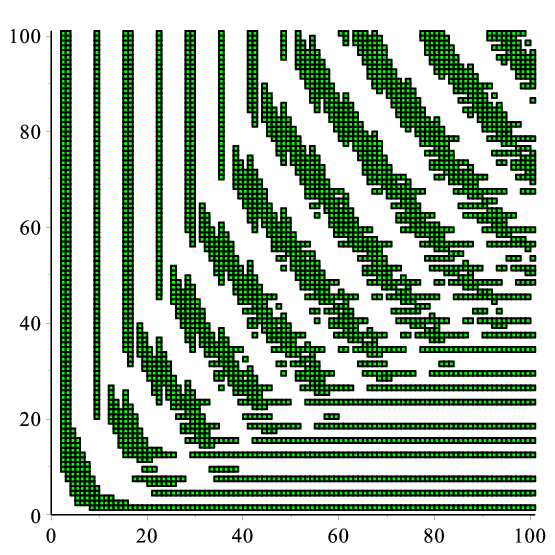}
\includegraphics[scale=0.2]{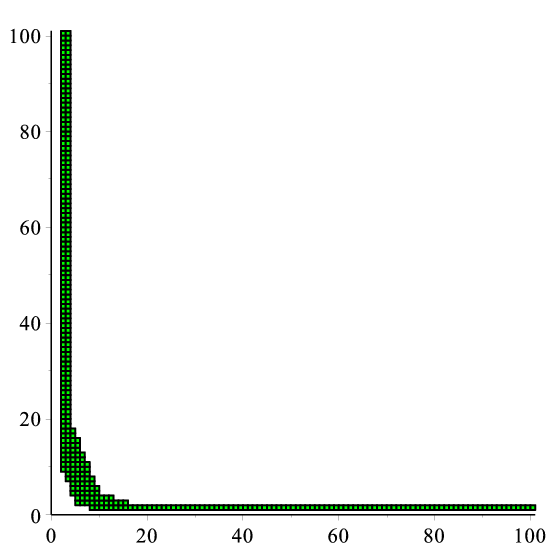}
\includegraphics[scale=0.2]{3a_miserstarstar}
\end{center}
\caption{The graphs show convergence after two iterations for the game $\M=\{(2,9),(3,7),(4,4),(5,2),(8,1)\}$, an example of case 3(a).}
\label{fig:3a}
\end{figure}

By comparison, the case 3(c) has most variation, and we do not yet know if all games in this class converge in a finite number of steps. We conclude by showing behavior of four games of the form $$\M_x=\{(0,5),(x,x),(5,0)\},$$ for $x=1,2,3,4$. Based on Figure~\ref{fig:3cx}, we hypothesize that  $\varphi(\M_1)=7,\varphi(\M_2)=6,\varphi(\M_3)=6,\varphi(\M_4)=5$. Note that some limit games have generalized `L-shapes', while others have `negative-slope-diagonal-stripes', and yet others appear to be a blend of the two. 

\begin{figure}[htb]

\includegraphics[scale=0.087]{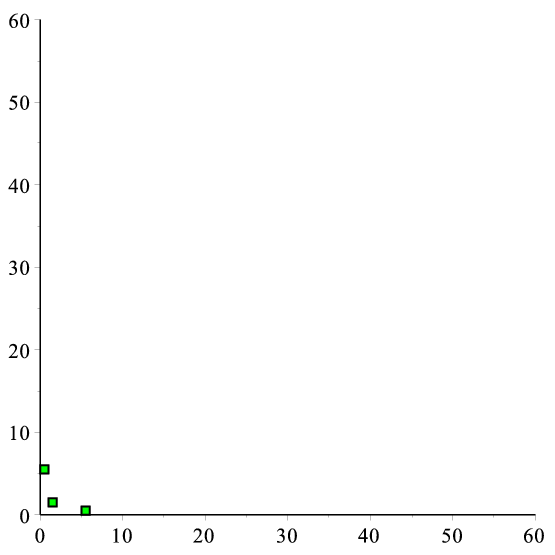}
\includegraphics[scale=0.087]{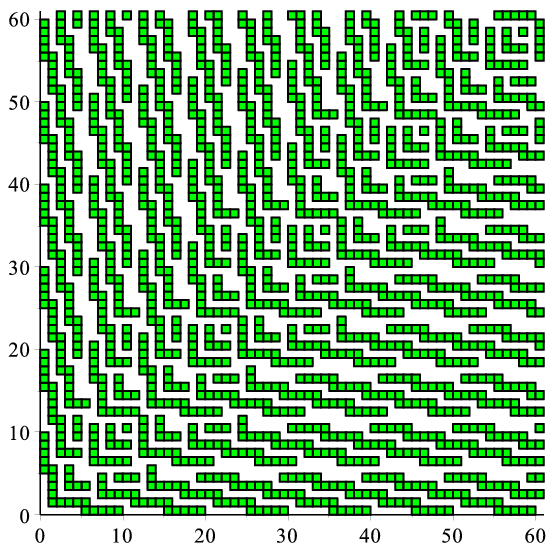}
\includegraphics[scale=0.087]{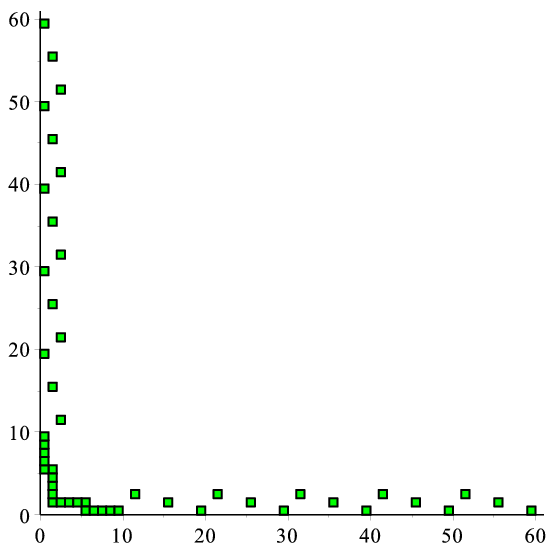}
\includegraphics[scale=0.087]{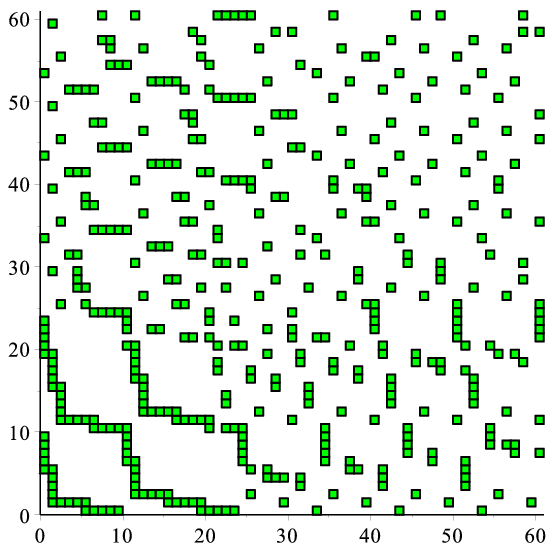}
\includegraphics[scale=0.087]{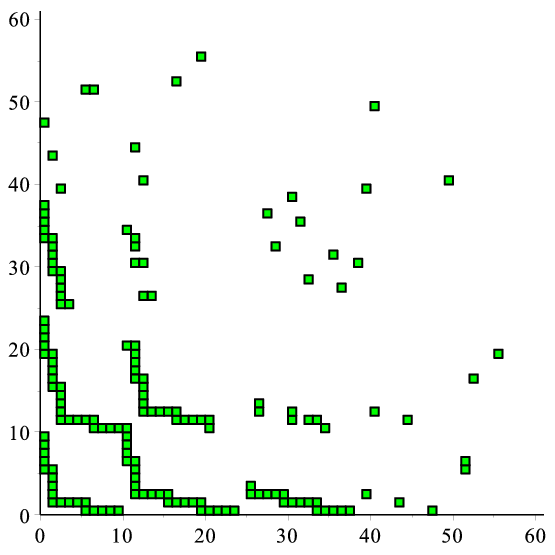}
\includegraphics[scale=0.087]{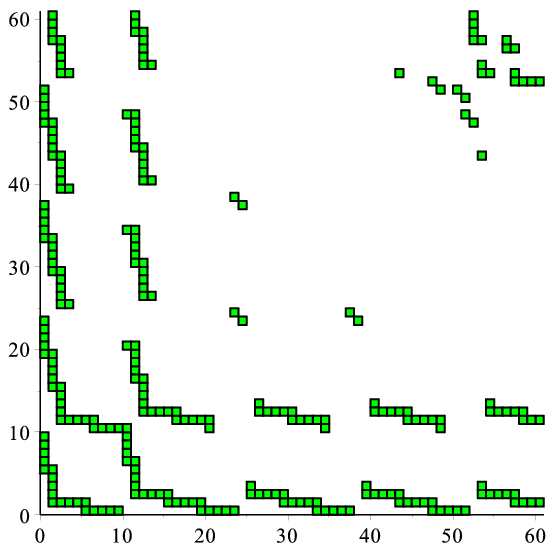}
\includegraphics[scale=0.087]{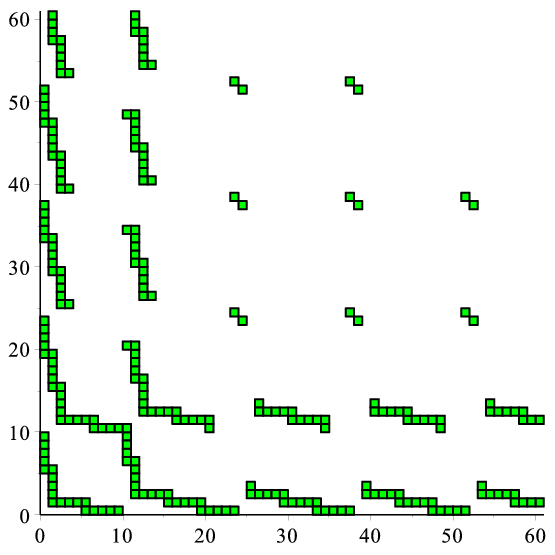}
\includegraphics[scale=0.087]{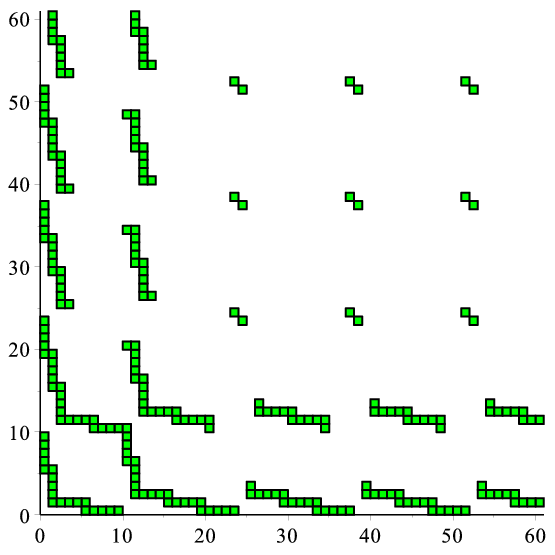}
\includegraphics[scale=0.087]{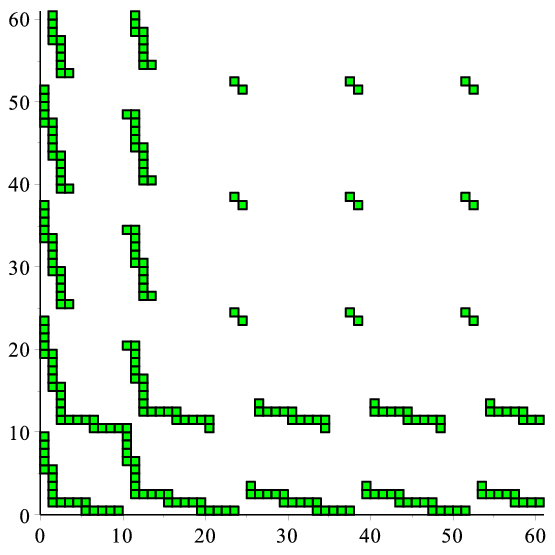}

\vspace{5 mm}
\includegraphics[scale=0.098]{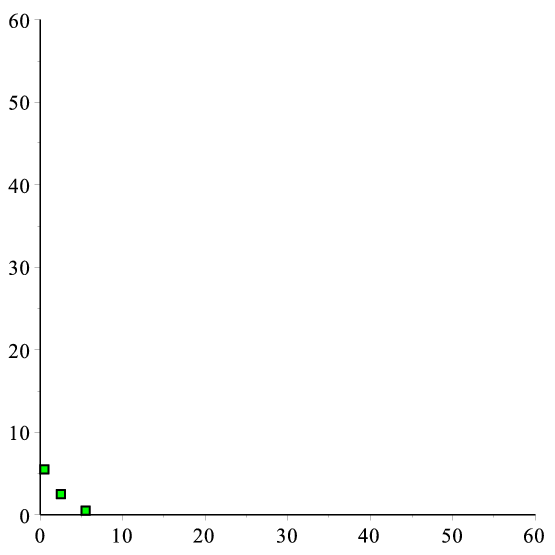}
\includegraphics[scale=0.098]{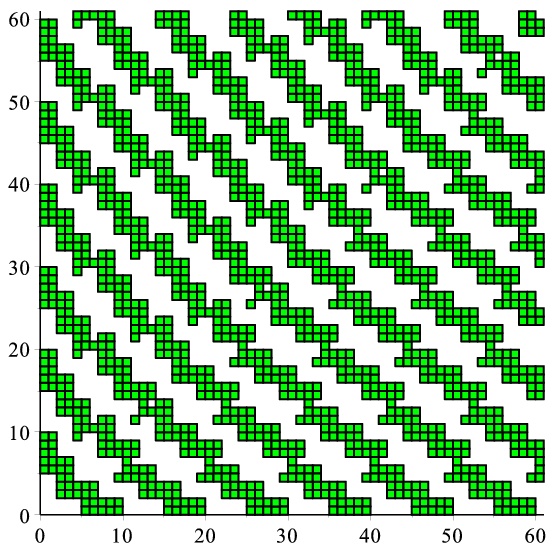}
\includegraphics[scale=0.098]{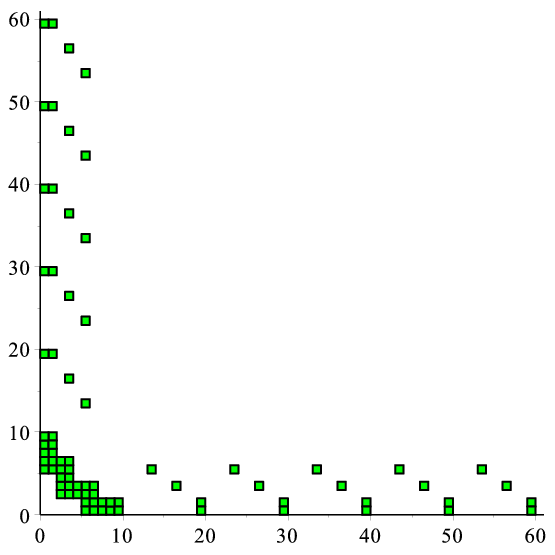}
\includegraphics[scale=0.098]{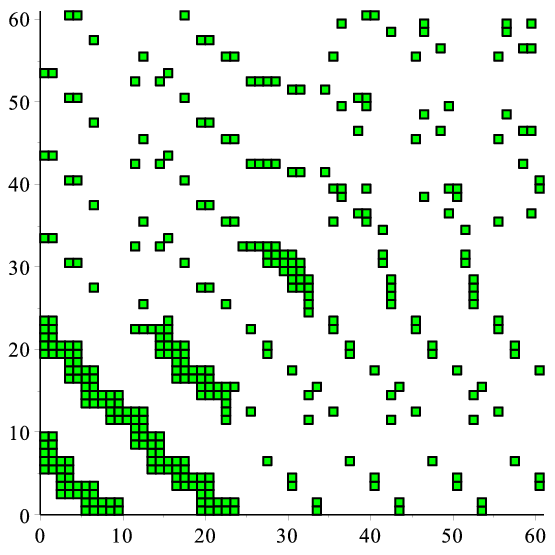}
\includegraphics[scale=0.098]{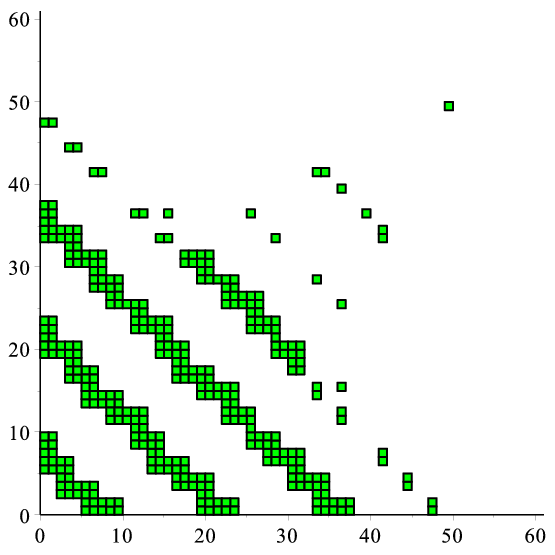}
\includegraphics[scale=0.098]{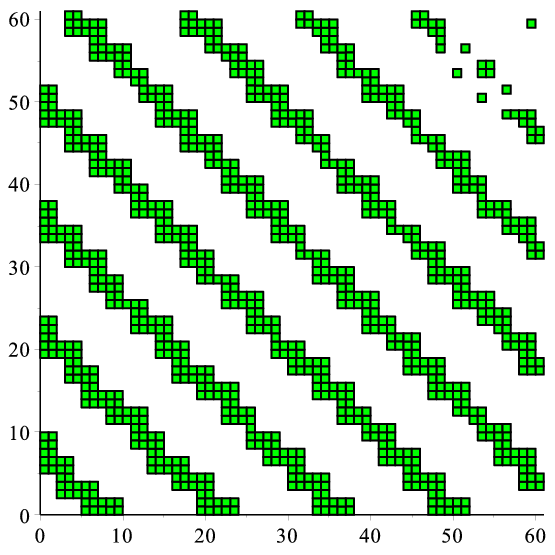}
\includegraphics[scale=0.098]{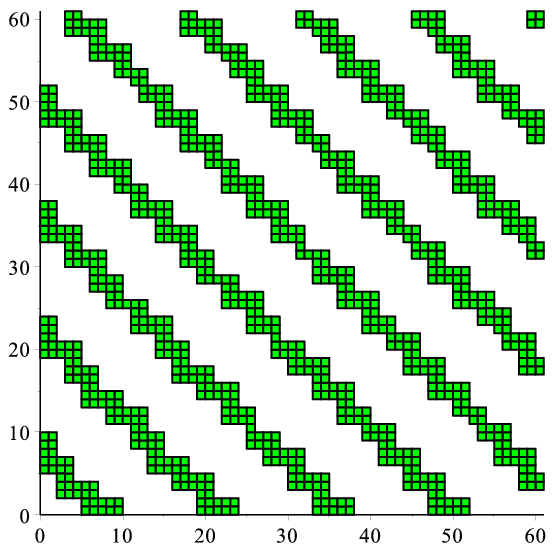}
\includegraphics[scale=0.098]{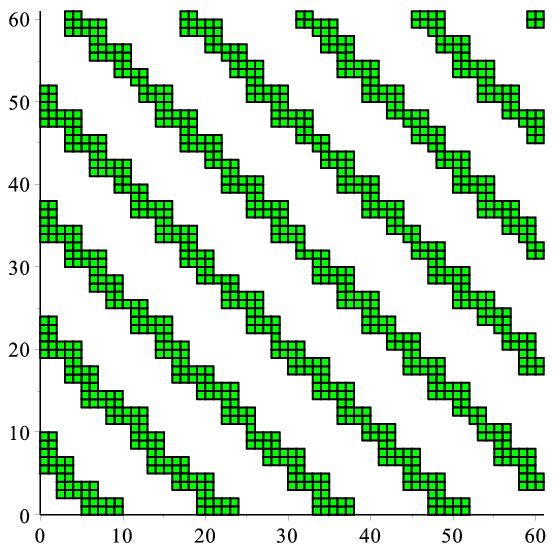}

\vspace{5 mm}
\includegraphics[scale=0.098]{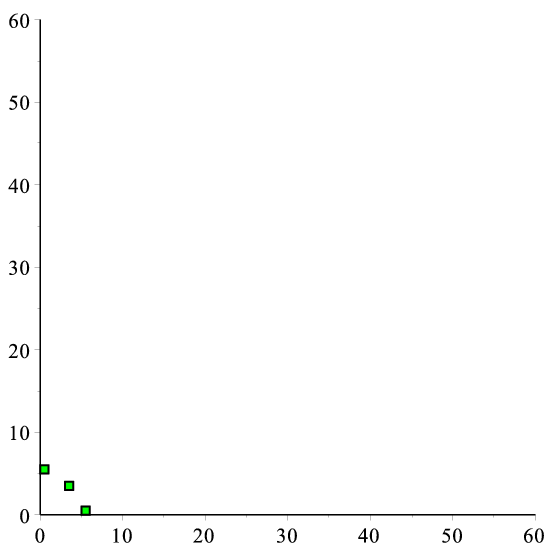}
\includegraphics[scale=0.098]{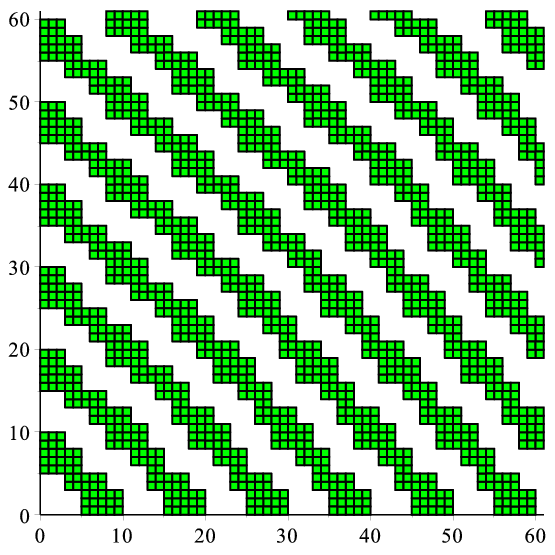}
\includegraphics[scale=0.098]{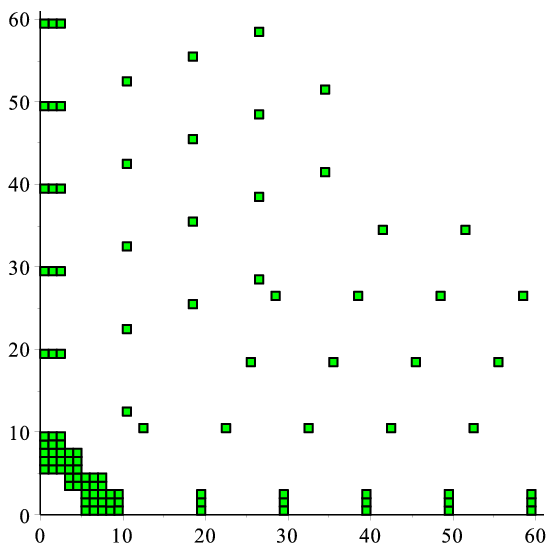}
\includegraphics[scale=0.098]{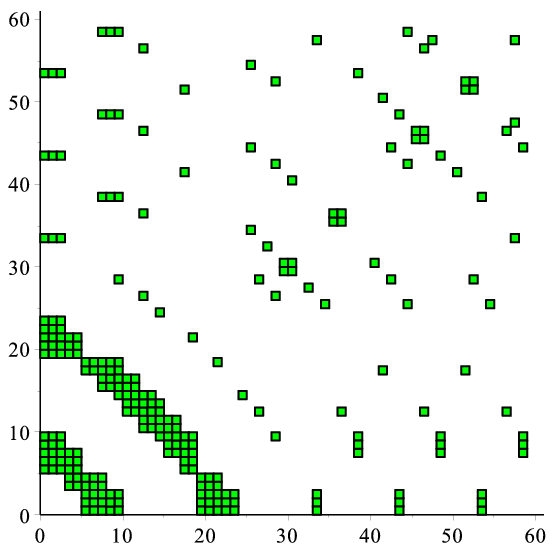}
\includegraphics[scale=0.098]{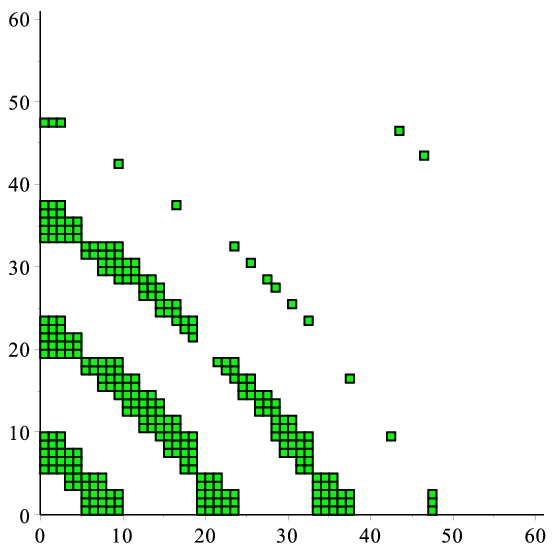}
\includegraphics[scale=0.098]{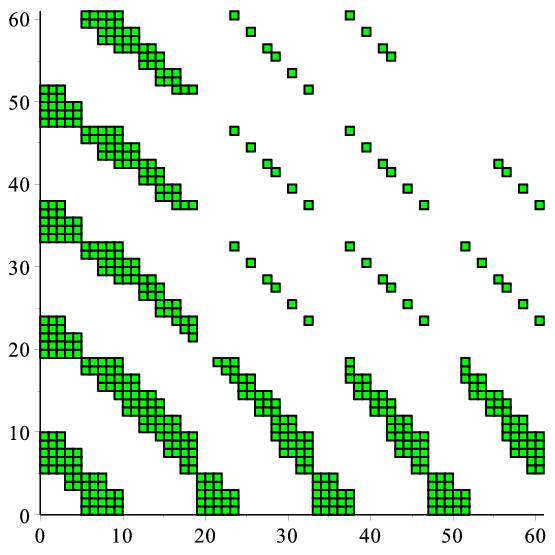}
\includegraphics[scale=0.098]{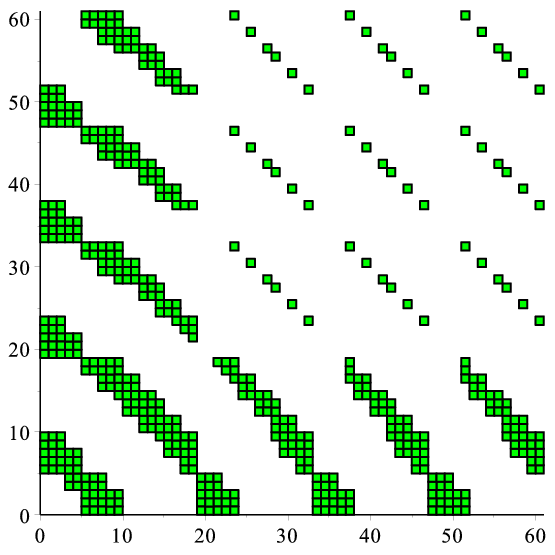}
\includegraphics[scale=0.098]{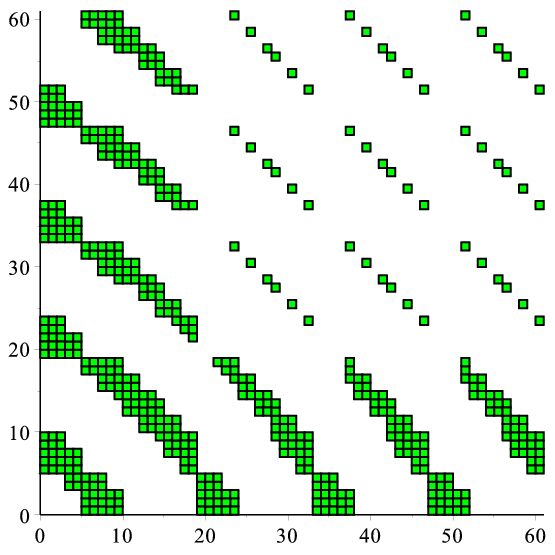}

\vspace{5 mm}
\includegraphics[scale=0.115]{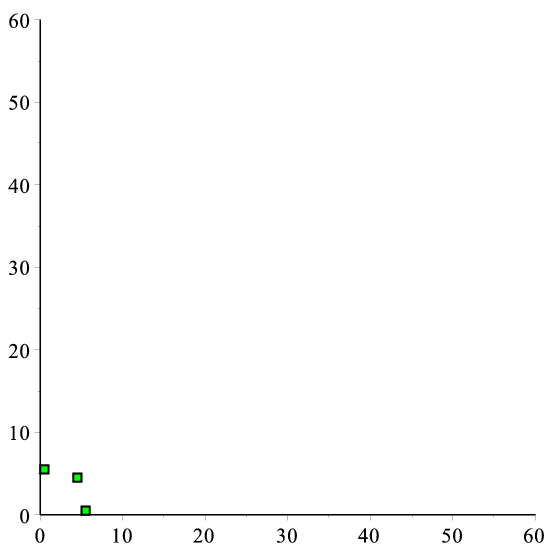}
\includegraphics[scale=0.115]{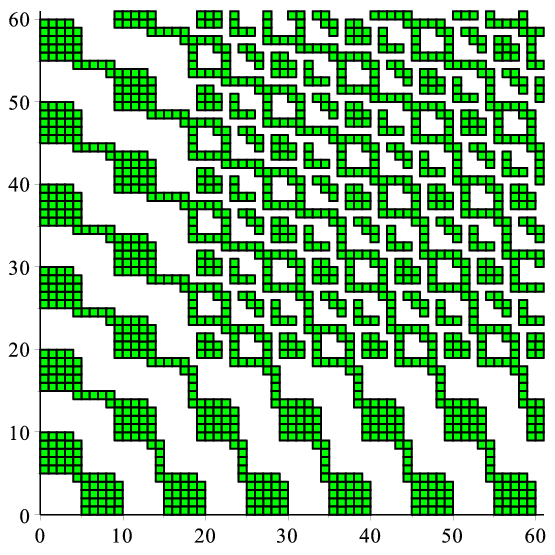}
\includegraphics[scale=0.115]{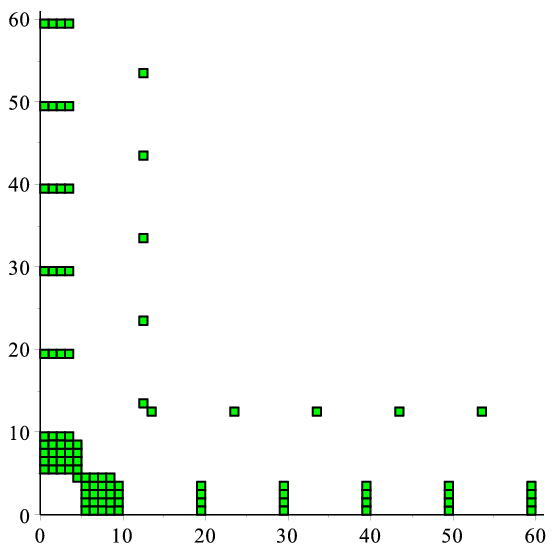}
\includegraphics[scale=0.115]{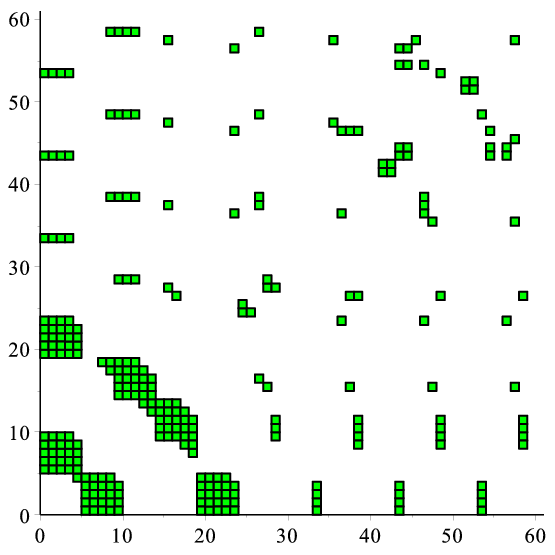}
\includegraphics[scale=0.115]{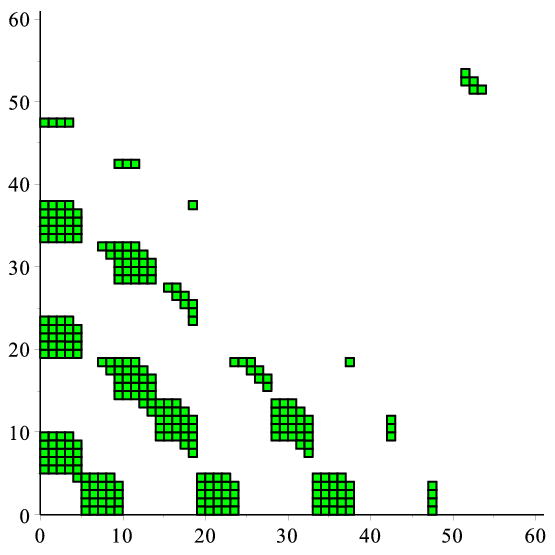}
\includegraphics[scale=0.115]{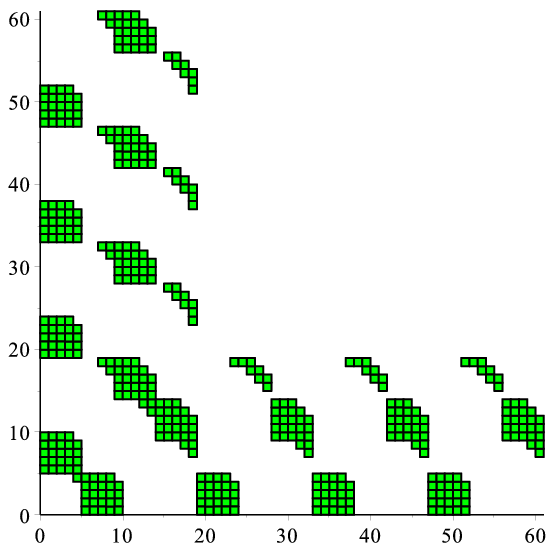}
\includegraphics[scale=0.115]{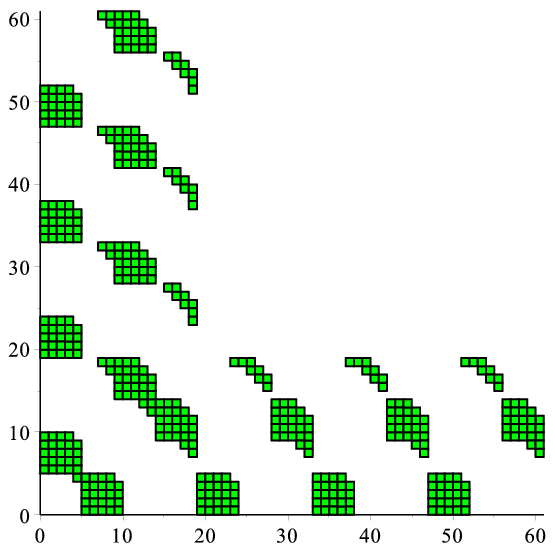}

\caption{Recurrence of the $\ms$-operator for 4 type 3c games $\M=\{(0,5),(x,x),(5,0)\}$, for $x=1,2,3,4$.}
\label{fig:3cx}
\end{figure}

 \newpage
The simplest non-trivial game whose limit game has `diagonal stripes of negative slopes' is $\M = \{(0,2),(1,1),(2,0)\}$. It converges in five steps to the game in Figure~\ref{fig:3csimple}. It generalizes the game $\{(0,1),(1,0)\}$, which trivially converges in one step to a checkerboard pattern.

\begin{figure}[h!]
\begin{center}
\includegraphics[scale=0.2]{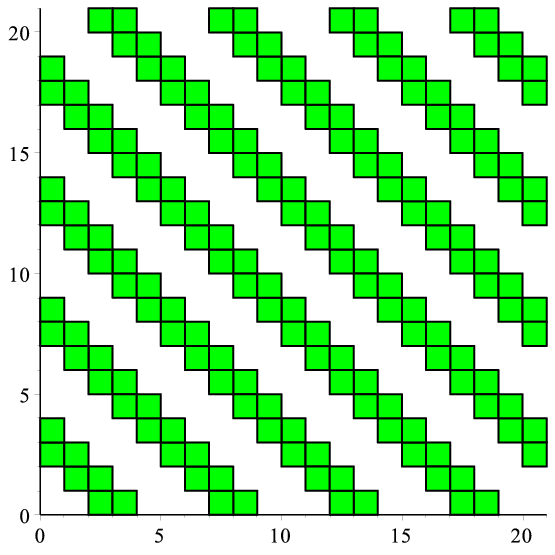}
\end{center}
\caption{A reflexive game with diagonal shaped moves.}
\label{fig:3csimple}
\end{figure}

We have performed many computer experiments in two dimensions, but have not (yet) been able to find any game for which the limit game behaves `randomly' or `chaotically'. This is quite different from reflexive games in normal-play, where the crystal-like patterns so common in mis\`re play are rare.

\end{document}